\documentclass[a4paper,fleqn]{cas-sc}

\usepackage[numbers, sort&compress]{natbib}

\usepackage{graphicx}
\usepackage{multirow}
\usepackage{booktabs}

\usepackage{amsmath, amssymb, amsfonts}
\usepackage{amsthm}
\usepackage{mathrsfs}

\usepackage[title]{appendix}
\usepackage{xcolor}
\usepackage{textcomp}
\usepackage{manyfoot}
\usepackage{algorithm}
\usepackage{algorithmicx}
\usepackage{algpseudocode}
\usepackage{listings}
\usepackage{float}

\theoremstyle{plain}
 
\newtheorem{lemma}{Lemma}[section]
\newtheorem{theorem}{Theorem}[section]
\newtheorem{corollary}[theorem]{Corollary}

\theoremstyle{definition}
\newtheorem{definition}{Definition}[section]
\newtheorem{example}{Example}[section]
\newtheorem{remark}{Remark}[section]
\newtheorem{prop}{Proposition}

\newcolumntype{L}{>{\centering\arraybackslash}m{1.8cm}}

\def\tsc#1{\csdef{#1}{\textsc{\lowercase{#1}}\xspace}}
\tsc{WGM}
\tsc{QE}
\tsc{EP}
\tsc{PMS}
\tsc{BEC}
\tsc{DE}
\begin{document}
\let\WriteBookmarks\relax
\def\floatpagepagefraction{1}
\def\textpagefraction{.001}
\shorttitle{Minimal monomial basis for Birkhoff interpolation}
\shortauthors{Y. Li, X. Jiang and Z. Li}
%\begin{frontmatter}

\title [mode = title]{Computing the minimal monomial basis for multivariate Birkhoff interpolation}                      
%\tnotemark[1,2]

%\tnotetext[1]{This document is the results of the research project funded by the National Science Foundation.}

%\tnotetext[2]{The second title footnote which is a longer text matter to fill through the whole text width and overflow into another line in the footnotes area of the first page.}

\author{Yuanhe Li}%[type=editor,
                      %  auid=000,bioid=1,
                      %  prefix=Sir,
                      %  role=Researcher,
                      %  orcid=0000-0001-0000-0000]
%\cormark[1]
%\fnmark[1]
\ead{13782387477@163.com}
%\ead[url]{www.jkkrishnan.in}

%\credit{Conceptualization of this study, Methodology, Software}

\affiliation{organization={School of Mathematics and Statistics, Changchun University of Science and Technology},
            %addressline={},
            city={Changchun},
            postcode={130000},
            country={China}}
%\affiliation[1]{organization={Department of Physics, J.K. Institute of Science},
               % addressline={Jawahar Nagar}, 
              %  city={Trivandrum},
%               citysep={}, % Uncomment if no comma needed between city and postcode
              %  postcode={695013}, 
              %  state={Kerala},
              %  country={India}}

\author{Xue Jiang}%[style=chinese]
\cormark[1]
%\author[2,3]{William {J. Hansen}}[%role=Co-ordinator, suffix=Jr,]
%\fnmark[2]
\ead{littledonna@163.com}
%\ead[URL]{https://www.university.org}

%\credit{Data curation, Writing - Original draft preparation}

%\affiliation[2]{organization={World Scientific University},
           %     addressline={Street 29}, 
           %     postcode={1011 NX}, 
           %      postcodesep={}, 
           %      city={Amsterdam},
           %      country={The Netherlands}}

\author{Zhe Li}
% \cormark[2]
% \fnmark[1,3]
 \ead{zheli200809@163.com}
% \ead[URL]{www.campus.in}

% \affiliation[3]{organization={University of Intelligent Studies},
             %    addressline={Street 15}, 
              %   city={Jabaldesh},
              %   postcode={825001}, 
             %    state={Orissa}, 
             %    country={India}}

\cortext[cor1]{Corresponding author}
%\cortext[cor2]{Principal corresponding author}
%\fntext[fn1]{This is the first author footnote, but is common to third author as well.}
%\fntext[fn2]{Another author footnote, this is a very long footnote and it should be a really long footnote. But this footnote is not yet sufficiently long enough to make two lines of footnote text.}

%\nonumnote{This note has no numbers. In this work we demonstrate $a_b$the formation Y\ _1 of a new type of polariton on the interface  between a cuprous oxide slab and a polystyrene micro-sphere placed on the slab.}

\begin{abstract}
This paper studies algorithms for computing the minimal monomial basis for multivariate Birkhoff interpolation problems. Our approach is built around the notion of a reverse reduced set, which serves as the key tool for bridging the interpolation conditions to the monomial basis, thereby avoiding the construction and evaluation of Vandermonde matrices required in existing algorithms. For the single-node case, we prove that after Gaussian elimination on the incidence matrix, the least monomials of the polynomial set corresponding to the interpolation conditions precisely constitute the minimal monomial basis. To handle the multi-node case, we exploit the one-to-one correspondence between interpolation functionals and formal power series, whereby interpolation conditions at arbitrary nonzero nodes can all be converted to the origin. This provides both a coherent theoretical framework and a constructive algorithm for determining a proper minimal monomial basis for the general multivariate Birkhoff interpolation problem. Numerical examples demonstrate the effectiveness of the proposed algorithm.

%\noindent\texttt{\textbackslash begin{abstract}} \dots 
%\texttt{\textbackslash end{abstract}} and
%\verb+\begin{keyword}+ \verb+...+ \verb+\end{keyword}+ 
%which contain the abstract and keywords respectively. 
%Each keyword shall be separated by a \verb+\sep+ command.
\end{abstract}

%\begin{graphicalabstract}
%\includegraphics{figs/cas-grabs.pdf}
%\end{graphicalabstract}

%\begin{highlights}
%\item The notion of a reverse reduced set bridges the interpolation conditions to the minimal monomial basis, enabling a purely algebraic approach. 
%\item For the single-node Birkhoff interpolation, the minimal monomial basis is obtained directly by Gaussian elimination on the incidence matrix.
%\item For the general multi-node case, interpolation conditions at arbitrary nodes are reduced to the origin via formal power series, yielding a constructive algorithm that avoids Vandermonde matrices and functional evaluations.
%\end{highlights}

\begin{keywords}
Multivariate Birkhoff interpolation \sep Formal power series \sep Minimal monomial basis
%quadrupole exciton \sep polariton \sep \WGM \sep \BEC
\end{keywords}

\raggedbottom

\maketitle

%%%%%%%%%%%%%%%%%%%%%%%%%%%%%%%%%%%%%%%%%%%%%%%%%%%%%%%%%%%%%%%%%%%%%%%%%%%%%%%%%%%%%%%%%%%%%%%%%%%%%%%%%%%%%%%%%%%%
%\raggedbottom

\section{Introduction}
As a fundamental problem, interpolation lies at the heart of computational mathematics and serves as a key tool in many of its subfields. The classical polynomial interpolation problem can be stated as follows: given a set of mutually distinct nodes and prescribed values of the function and its derivatives up to certain orders at each node, find a polynomial of minimal degree that satisfies all the given interpolation conditions.

Polynomial interpolation problems can be classified into ideal interpolation and non-ideal interpolation, according to the distribution of derivative orders in the interpolation conditions. In ideal interpolation, the conditions start from the zeroth-order derivative and proceed consecutively upward. Moreover, the total number of conditions matches the dimension of the polynomial space. Typical examples are Lagrange interpolation and Hermite interpolation.
In contrast, non-ideal interpolation breaks the consecutive constraint from order zero. Certain intermediate derivative orders may be skipped. This leads to a non-consecutive configuration of conditions, which makes both regularity analysis and the computation of proper interpolation bases considerably more difficult. As a prototypical example of non-ideal interpolation, the Birkhoff interpolation problem embodies the difficulties described above. In recent years, it has attracted considerable attention from many researchers, leading to a series of important results.

The Birkhoff multivariate interpolation scheme $(Z, S, E)$ consists of three components: the set of interpolation nodes $Z$, the interpolation space $S$, and the incidence matrix $E$. Atkinson and Sharma proposed the Polya's conditions for the incidence matrix corresponding to the multivariate Birkhoff interpolation problem in \cite{wen11}. Bojanov and Xu analyzed the Birkhoff interpolation problem in the two-dimensional case, where the nodes were distributed in concentric circles \cite{wen12}.
Crainic carried out a systematic investigation of multivariate Birkhoff interpolation schemes. For uniform and Cartesian (rectangular) node configurations, they derived necessary and sufficient conditions for the existence and almost regularity of such schemes \cite{21,25}, examined restrictions on the number of mixed derivatives \cite{22}. In \cite{26}, the authors analyzed interpolation schemes that involve only few derivatives.

Dell'Accio and Tommaso considered Hermite--Birkhoff interpolation on scattered data under completeness conditions, using Shepard basis functions combined with local polynomials that interpolate at triangle vertices \cite{Shepard}. Birkhoff interpolation problems are not always solvable, even in appropriate polynomial or rational spaces. The approach of Dell'Accio et al. \cite{data} splits the problem into uniquely solvable polynomial subproblems and uses multinode rational basis functions to construct a global interpolant.
Shepard interpolation has advanced significantly with the development of Shepard operators and now plays a key role in function approximation, surface reconstruction, and related fields. A detailed discussion can be found in the comprehensive survey \cite{27}.  For Birkhoff data, solutions in rational function spaces can likewise be obtained by linearly combining appropriately chosen rational basis functions with polynomials that interpolate on local data \cite{Shepard,data}. An alternative approach is to design a suitable family of basis functions tailored for Birkhoff data \cite{43}, which is particularly well suited to scattered data on the sphere and on other manifolds \cite{44}.

The algorithm proposed by Chai et al. in \cite{lei} allowed the computation of the minimal degree interpolation polynomial for multivariate interpolation problems. 
Building on this foundation, Lei et al. further developed a fast algorithm for computing the minimal monomial basis of multivariate Birkhoff interpolation \cite{1656}. The interpolation problems addressed by these two algorithms involve only monomial interpolation conditions.
Cui and Lei \cite{stable} generalized the multivariate Birkhoff interpolation problem, gave conditions for the existence of solutions, and proposed an algorithm for computing monomial bases that remains stable even when the nodes contain errors. Xia et al. \cite{p.xia} proposed linearization methods for univariate Birkhoff rational interpolation and validated them through numerical experiments. In \cite{C2}, Xia et al. constructed new rational weight functions to raise the smoothness of the weighted patchwise interpolation from $C^2$ to $C^M$ ($2 \leq M \leq 20$). This approach was extended to the rational interpolation setting, yielding better approximation results. The method of \cite{45} studied univariate Birkhoff rational interpolation via parametric linear systems of Cauchy and osculatory rational interpolation. It was later generalized to the multivariate case \cite{46}.

In recent years, Rhouni et al. proposed the RHBPIA algorithm \cite{pia}, which reformulated the univariate Birkhoff interpolation problem via Schur complement theory and employed a generalized recursive strategy to compute the classical univariate Birkhoff interpolation polynomial.  
Subsequently, building on \cite{GRPMIA}, Rhouni et al. further developed the RHBMPIA algorithm \cite{mpia} for solving the Birkhoff matrix polynomial interpolation problem in a general framework. 

In \cite{ideal}, Jiang and Gong proposed the concept of reverse reduced set to compute the  minimal monomial basis  and the reduced Gröbner basis of the associated vanishing ideal for ideal interpolation. This paper generalizes the idea of reverse reduced set to the Birkhoff interpolation. 
Compared with existing methods  \cite{lei,stable} for computing the minimal monomial basis for Birkhoff interpolation, 
our algorithm does not depend on Vandermonde matrices and requires no evaluation of interpolation functionals.
For the single-node case, we prove that the computation of the minimal interpolation basis can be reduced directly to row reduction of the corresponding incidence matrix. In the multi-node case, we use formal power series to convert the interpolation conditions at each nonzero node into conditions at the origin, and apply Gaussian elimination to the truncated monomial coefficients to obtain the minimal monomial interpolation basis with respect to  $\prec$.

The paper is organized as follows. Section~\ref{sec:2} presents the definition of the multivariate Birkhoff interpolation problem together with the necessary preliminaries, including the concept of reverse reduced set and their relevant properties. In Section~\ref{sec:3}, we present a method for computing the minimal monomial basis directly from the interpolation condition polynomials using reverse reduced set. Section~\ref{sec:4} extends the method to the multi-node case and presents an algorithm for computing the minimal monomial interpolation basis w.r.t. $\prec$ for the general multivariate Birkhoff interpolation problem. The effectiveness of the proposed methods is illustrated by concrete examples at the end of Sections~\ref{sec:3} and~\ref{sec:4}.

\section{Preliminary}\label{sec:2}
In this paper, let $\mathbb{F}$ denote either the field of real numbers $\mathbb{R}$ or the field of complex numbers $\mathbb{C}$, let 
$\mathbb{N}$ denote the set of non-negative integers, and let $\mathbb{F}[x_1,x_2,\dots,x_n]$ denote the ring of polynomials in $n$ variables over the field $\mathbb{F}$. The Birkhoff interpolation problem in $n$ variablescan be described as follows:
Given a set of interpolation nodes $Z = \{\boldsymbol{z}_i\}_{i=1}^m = \{(x_{i1}, x_{i2}, \dots, x_{in})\}_{i=1}^m \subset \mathbb{F}^n
$ and a monomial sequence $S = [\boldsymbol{x}^{\boldsymbol{\alpha}_1}, \boldsymbol{x}^{\boldsymbol{\alpha}_2}, \dots, \boldsymbol{x}^{\boldsymbol{\alpha}_l}]$ arranged in a prescribed order $\prec$. For $ i=1,2,\dots,l$, let ${D}^{\boldsymbol{\alpha}_i}$ denote the differential operator associated with the monomial $\boldsymbol{x}^{\boldsymbol{\alpha}_i}$. The incidence matrix is 
\begin{equation}\label{eq:E}
    \boldsymbol{E} = 
\begin{pmatrix}
E_1 \\
E_2 \\
\vdots \\
E_m
\end{pmatrix},
\end{equation}
where $E_i = (e_{jh}^i), \ 1 \leq i \leq m,   j = 1, 2, \dots, s_i,   h = 1, 2, \dots, l$. The matrix has no identically zero rows. That is, $e_{jh}^i \in \mathbb{F}$ denotes the coefficient of the $h$-th monomial (taken in the prescribed order $S$) in the $j$-th interpolation condition at node $\boldsymbol{z}_i$. Let $L_{i,j}$ be the interpolation functional induced by the incidence matrix $\boldsymbol{E}$, i.e.,
\[
L_{i,j} = \delta_{\boldsymbol{z}_i} \circ \sum_{h=1}^l e^i_{jh} D_h, \quad 1 \leq i \leq m, \ 1 \leq j \leq s_i.
\] 
Given arbitrary interpolation values $c_{ij} \in \mathbb{F}$, the Birkhoff
interpolation problem to find a basis such that in the space $\mathcal{P}_S$ spanned by this basis, there exists a unique polynomial $p(\boldsymbol{x})$ satisfying
\begin{equation*}
L_{i,j}(p(\boldsymbol{x})) = \sum_{h=1}^{l} e_{jh}^i D_h p(\boldsymbol{z}_i) = c_{ij}, \quad 1 \leq i \leq m, \ 1 \leq j \leq s_i.
\end{equation*}
Such a basis is called a proper interpolation basis for the interpolation problem, the space $\mathcal{P}_S$ is called the interpolation space, and the polynomial $p(\boldsymbol{x})$ is called the interpolation polynomial. 
\begin{remark}
Here $|N| = \dim \mathcal{P}_S = \sum_{i=1}^{m} s_i$ denotes the total number of rows of the incidence matrix $\boldsymbol{E}$, i.e., the total number of interpolation conditions. The interpolation conditions are uniquely determined by the incidence matrix $\boldsymbol{E}$ and the prescribed graded order $S$. 
\end{remark}

\begin{theorem}(\cite{stable})\label{rank}
The interpolation problem with conditions given by \eqref{eq:E} has a solution if and only if $\operatorname{rank}(E_i)=\operatorname{rank}(E_i,\boldsymbol{c}_i), i=1,2,...,m,$ where $\boldsymbol{c}_i=(c_{i1},c_{i2},\dots,c_{is_i})^{\mathrm{T}}$.
\end{theorem}
\begin{remark}
The compatibility condition $\operatorname{rank}(E_i) = \operatorname{rank}(E_i, c_i)$ given by the theorem \ref{rank} only guarantees the existence of a solution to the interpolation problem, it does not require the interpolation functionals to be linearly independent. %Here $c_i = (c_{i1}, c_{i2}, \dots, c_{i,s_i})^{\mathrm{T}}$ is the data vector at the node $\boldsymbol{z}_i$. 
Linear independence of the functionals is equivalent to the incidence matrix $E_i$ having full row rank, which is a stronger condition that ensures uniqueness of the solution. We focus on the case of linearly independent interpolation functionals, ensuring a unique interpolation polynomial for any given data. The construction of the monomial basis is studied within this uniqueness framework.
\end{remark}

\begin{definition}
Given a graded order $\prec$ and interpolation conditions $\Delta = \text{span}\{L_1, L_2, \dots, L_n\},$ where the $L_1, L_2, \dots, L_n$ are linearly independent functionals, a set of monomials $T$ is called a minimal monomial basis for $\Delta$ w.r.t. $\prec$ if
\begin{enumerate}
\item[(1)] $T$ is a proper interpolation basis for $\Delta$;
\item[(2)] Replacing any monomial in $T$ by a monomial that is lower in the order $\prec$ can never yield another proper interpolation basis for $\Delta$.
\end{enumerate}
\end{definition}

\begin{example}\label{ex1}
 We consider the bivariate case as an example. Let the set of interpolation nodes be $Z = \{\boldsymbol{z}_1, \boldsymbol{z}_2\} = \{(x_1, y_1), (x_2, y_2)\} \subset \mathbb{R}^2$. Given the sequence $S = [y, y^2, xy]$ ordered by the graded lexicographic order $y \prec_{\mathrm{grlex}} x$, the corresponding interpolation space is denoted by $\mathcal{P}_S$. The incidence matrix $
\boldsymbol{E} = 
\begin{pmatrix}
E_1 \\
E_2
\end{pmatrix},$ where $E_1 = 
\begin{pmatrix}
1 & 0 & 0 \\
1 & 0 & 1
\end{pmatrix},  E_2 = 
\begin{pmatrix}
1 & 0 & 0 \\
4 & 1 & 2
\end{pmatrix}.$ The corresponding sequence of differential operators is given by $D = 
\begin{bmatrix}
\frac{\partial}{\partial y}, \frac{\partial^2}{\partial y^2}, \frac{\partial^2}{\partial x \partial y}
\end{bmatrix}.$ Let $L_{1,1}, L_{1,2}, L_{2,1}, L_{2,2}$ denote the interpolation functional determined by the conditions above. Given arbitrary interpolation values $\{c_{11}, c_{12}, c_{21}, c_{22}\}$, the Birkhoff
interpolation problem is to find a basis for the interpolation space $\mathcal{P}_S$ such that there exists a unique polynomial $p(\boldsymbol{x}) \in \mathcal{P}_S$ satisfying the interpolation conditions
\begin{align*}
L_{1,1}(p(\boldsymbol{x})) &= 1 \cdot \frac{\partial}{\partial y} p(\boldsymbol{z}_1) = c_{11}, \\
L_{1,2}(p(\boldsymbol{x})) &= 1 \cdot \frac{\partial}{\partial y} p(\boldsymbol{z}_1) + 1 \cdot \frac{\partial^2}{\partial x \partial y} p(\boldsymbol{z}_1) = c_{12}, \\
L_{2,1}(p(\boldsymbol{x})) &= 1 \cdot \frac{\partial}{\partial y} p(\boldsymbol{z}_2) = c_{21}, \\
L_{2,2}(p(\boldsymbol{x})) &= 4 \cdot \frac{\partial}{\partial y} p(\boldsymbol{z}_2) + 1 \cdot \frac{\partial^2}{\partial y^2} p(\boldsymbol{z}_2) + 2 \cdot \frac{\partial^2}{\partial x \partial y} p(\boldsymbol{z}_2) = c_{22}.
\end{align*}
\end{example}
For computational convenience, we denote the interpolation conditions determined by the incidence matrix $\boldsymbol{E}$ and the sequence $S$ in Example~$\ref{ex1}$ as follows:
\[
\Delta = 
\begin{cases} 
\delta_{\boldsymbol{z}_1} \circ \text{span}\{D_y, D_y + D_x D_y\}, \\
\delta_{\boldsymbol{z}_2} \circ \text{span}\{D_y, 4D_y + D_y^2 + 2D_x D_y\}.
\end{cases}
\]

\begin{definition}[Multi-index factorial]
For the general $n$-variate case, let $\boldsymbol{\alpha} = (\alpha_1, \alpha_2, \dots, \alpha_n) \in \mathbb{N}^n$ be a multi-index. The factorial of $\boldsymbol{\alpha}$ is defined as the product of the factorials of its components:
\[
\boldsymbol{\alpha}! = \alpha_1! \, \alpha_2! \cdots \alpha_n! = \prod_{i=1}^n \alpha_i!.
\]
The length of $\boldsymbol{\alpha}$ is defined by
$|\boldsymbol{\alpha}| = \alpha_1 + \alpha_2 + \cdots + \alpha_n.$
The differential operator $D^{\boldsymbol{\alpha}}$ associated with $\boldsymbol{\alpha}$ is given by
\[
D^{\boldsymbol{\alpha}} f = \frac{\partial^{|\boldsymbol{\alpha}|} f}{\partial x_1^{\alpha_1} \partial x_2^{\alpha_2} \cdots \partial x_n^{\alpha_n}}.
\]
\end{definition}

Let $\mathbb{F}[[\boldsymbol{x}]]$ denote the ring of formal power series in the variables $x_1,x_2, \dots, x_n$, i.e., \[
\mathbb{F}[[\boldsymbol{x}]] := \mathbb{F}[x_1, x_2, \dots, x_n] = \left\{ \sum_{\boldsymbol{\alpha} \in \mathbb{N}^n} c_{\boldsymbol{\alpha}} \boldsymbol{x}^{\boldsymbol{\alpha}} : c_{\boldsymbol{\alpha}} \in \mathbb{F} \right\},
\]
where \(\boldsymbol{\alpha} = (\alpha_1, \alpha_2, \dots, \alpha_n)\), \(\boldsymbol{x} ^{\boldsymbol{\alpha} } = x_1^{\alpha_1} x_2^{\alpha_2} \cdots x_n^{\alpha_n}\). For any $\lambda \in \mathbb{F}[[\boldsymbol{x}]]$, define
\[
\lambda * f := (\lambda(D)f)(\boldsymbol{0}) = \sum_{\boldsymbol{\alpha} \in \mathbb{N}^n} \frac{D^{\boldsymbol{\alpha}} \lambda(\boldsymbol{0}) \cdot D^{\boldsymbol{\alpha}} f(\boldsymbol{0})}{\boldsymbol{\alpha}!}, \quad \forall f \in \mathbb{F}[\boldsymbol{x}].
\]
Taking $\lambda = e^{\boldsymbol{x} \cdot \boldsymbol{z}} \in \mathbb{F}[[\boldsymbol{x}]]$, we have
\[
e^{\boldsymbol{x} \cdot \boldsymbol{z}} * f = \sum_{\boldsymbol{\alpha} \in \mathbb{N}^n} \frac{D^{\boldsymbol{\alpha}} f(\boldsymbol{0})}{\boldsymbol{\alpha}!} \boldsymbol{z}^{\boldsymbol{\alpha}} = f(\boldsymbol{z}) = \delta_{\boldsymbol{z}} f, \quad \forall f \in \mathbb{F}[\boldsymbol{x}].
\] 
It is easy to verify that the one-to-one correspondence between \( (\mathbb{F}[\boldsymbol{x}])' \) and \( \mathbb{F}[[\boldsymbol{x}]] \) is as shown in Table \ref{biao1}.
%\newpage
\begin{table}[width=.9\linewidth,cols=6,pos=h]
\caption{the one-to-one correspondence between \( (\mathbb{F}[\boldsymbol{x}])' \) and \( \mathbb{F}[[\boldsymbol{x}]] \)}\label{biao1}
\begin{tabular*}{\tblwidth}{@{} LLLLLL@{} }
\toprule
\(\hat{\lambda} \in (\mathbb{F}[\boldsymbol{x}])'\) & 
\(\delta_{\boldsymbol{z}}\) & 
\(\delta_{\boldsymbol{0}}\) & 
\(\delta_{\boldsymbol{z}} \circ D_x D_y\) & 
\(\delta_{\boldsymbol{0}} \circ D_x\) & 
\(\cdots\) \\
\midrule
\(\lambda \in \mathbb{F}[[\boldsymbol{x}]]\) & 
\(e^{\boldsymbol{x} \cdot \boldsymbol{z}}\) & 
\(1\) & 
\(xy e^{\boldsymbol{x} \cdot \boldsymbol{z}}\) & 
\(x\) & 
\(\cdots\) \\
\bottomrule
\end{tabular*}
\end{table}

For example, \(\hat{\lambda}_1 = \delta_{(1,2)} \circ D_y^2\) can be represented as \(\lambda_1 = y^2e^{x+2y}\), which yields \(\lambda_1(D) = e^{D_x + 2D_y} D_y^2\). 
\begin{prop}
With the above notations,
\begin{equation}\label{eq2}
\delta_{\boldsymbol{z}} \circ P(D) = \delta_{\boldsymbol{0}} \circ e^{\boldsymbol{z}D} P(D).
\end{equation}
\end{prop}
\begin{proof}
    For any polynomial $f$, \[
\delta_{\boldsymbol{z}} \circ P(D)(f) = P(\boldsymbol{x})e^{\boldsymbol{x}\cdot\boldsymbol{z}} * f ,
\]
\[
\delta_{\boldsymbol{0}} \circ e^{zD}P(D)(f) = e^{\boldsymbol{x}\cdot\boldsymbol{z}}P(\boldsymbol{x}) * f ,
\]
this indicates that $\delta_{\boldsymbol{z}} \circ P(D)(f) = \delta_{\boldsymbol{0}} \circ e^{\boldsymbol{z}D}P(D)(f),$
which completes the proof.
\end{proof}

\section{Computing the proper interpolation basis for single-node Birkhoff interpolation}\label{sec:3}

The single-node Birkhoff interpolation problem can be stated more concisely as follows: Denote the interpolation node by $\boldsymbol{z}$, given a sequence of monomials $[\boldsymbol{x}^{\boldsymbol{\alpha}_1}, \boldsymbol{x}^{\boldsymbol{\alpha}_2}, \dots, \boldsymbol{x}^{\boldsymbol{\alpha}_l}]$ ordered according to a graded order $\prec$, denote by $D^{\boldsymbol{\alpha}_h}$ ($h = 1, 2, \dots, l$) the differential operator associated with $\boldsymbol{x}^{\boldsymbol{\alpha}_h}$. Suppose that there are $N$ linearly independent interpolation conditions. Then the associated incidence matrix, as defined in (\ref{eq:E}), takes the form $E = (e_{jh})_{N \times l}, \  e_{jh} \in \mathbb{R}.$ For each $j = 1, 2, \dots, N$, let $L_j$ denote the interpolation functional determined by the conditions above, i.e.,\[
L_j = \delta_{\boldsymbol{z}} \circ \sum_{h=1}^l e_{jh} D^{\boldsymbol{\alpha}_h}, \quad 1 \leq j \leq N.
\]
For any given interpolation values $c_1, c_2, \dots, c_N$, find an interpolation polynomial $p(\boldsymbol{x})$ such that
\begin{equation*}
 L_j(p(\boldsymbol{x})) = \sum_{h=1}^l e_{jh} D^{\boldsymbol{\alpha}_h} p(\boldsymbol{z}) = c_j,\ 1\leq j\leq N.
\end{equation*}
To simplify notation, the interpolation conditions can be expressed as
\begin{equation}
\Delta = \delta_{\boldsymbol{z}} \circ \operatorname{span}\{P_1(D), P_2(D), \dots, P_N(D)\}, \label{eq:2.3}
\end{equation}
where $P_j(\boldsymbol{x}) = \sum_{h=1}^{l} e_{jh} \boldsymbol{x}^{\boldsymbol{\alpha}_h} \in \mathbb{F}[\boldsymbol{x}]$, $j = 1, 2, \dots, N$, are $N$ linearly independent polynomials.

\begin{definition}
The support of the polynomials $P_1(\boldsymbol{x}), P_2(\boldsymbol{x}), \dots, P_N(\boldsymbol{x})$, denoted by $\Lambda\{P_1(\boldsymbol{x}), P_2(\boldsymbol{x}), \dots, P_N(\boldsymbol{x})\}$, is the set of all monomials appearing with nonzero coefficients in $P_1(\boldsymbol{x}),P_2({\boldsymbol{x}}) \dots, P_N(\boldsymbol{x})$.

\end{definition}
For example, taking $P_1(\boldsymbol{x}) = y + y^2 + xy + 2y^3$ and $P_2(\boldsymbol{x}) = 4y + 5y^2 + 6x^2y + xy^2$, we have
\[
\Lambda\{P_1(\boldsymbol{x}), P_2(\boldsymbol{x})\} = \{y, y^2, xy, y^3, xy^2, x^2y\}.
\]

\begin{definition}
Given an arbitrary graded order $\prec$, a set of linearly independent polynomials
\[
\{P_1(\boldsymbol{x}), P_2(\boldsymbol{x}), \dots, P_N(\boldsymbol{x})\} \subset \mathbb{F}[\boldsymbol{x}]
\]
is called a reverse reduced set w.r.t. $\prec$ if the following conditions hold:
\begin{enumerate}
    \item For each $i = 1, 2, \dots, N$, let $\operatorname{lm}(P_i(\boldsymbol{x}))$ denote the least monomial of $P_i(\boldsymbol{x})$ w.r.t. $\prec$, and assume that its coefficient is~$1$;
    
    \item For any distinct indices $i, j \in \{1, 2, \dots, N\}$, the monomial $\operatorname{lm}(P_i(\boldsymbol{x}))$ does not appear in $P_j(\boldsymbol{x})$, i.e.,
    \[
    \operatorname{lm}(P_i(\boldsymbol{x})) \notin \Lambda\{P_j(\boldsymbol{x})\}, \quad i \neq j.
    \]
\end{enumerate}
\end{definition}
For example, with respect to the graded lexicographic order $y \prec_{\mathrm{grlex}} x$, both
\[
\{1,\ y + xy,\ x\}, \quad \{y,\ y^2 + y^3,\ x^2 + xy\} 
\]
are reverse reduced sets.

\begin{lemma}
\label{lem:support_intersection}
Consider the interpolation conditions of a Birkhoff interpolation problem at the origin:
\[
\Delta_{\boldsymbol{0}} = \delta_{\boldsymbol{0}} \circ \operatorname{span}\{P_1(D), P_2(D), \dots, P_N(D)\},
\]
where $P_1(\boldsymbol{x}), P_2(\boldsymbol{x}), \dots, P_N(\boldsymbol{x})$ are linearly independent polynomials in $\mathbb{F}[\boldsymbol{x}]$. Let $T = \{\boldsymbol{x}^{\boldsymbol{\beta}_1}, \boldsymbol{x}^{\boldsymbol{\beta}_2}, \dots, \boldsymbol{x}^{\boldsymbol{\beta}_N}\}$ be an interpolation monomial basis for $\Delta_{\boldsymbol{0}}$. Then for each $i = 1, 2, \dots, N$, the support $\Lambda\{P_i(\boldsymbol{x})\}$ has a nonempty intersection with $T$.
\end{lemma}

\begin{proof}
We proceed by contradiction. Suppose that $T = \{\boldsymbol{x}^{\boldsymbol{\beta}_1}, \boldsymbol{x}^{\boldsymbol{\beta}_2}, \dots, \boldsymbol{x}^{\boldsymbol{\beta}_N}\}$ is an interpolation monomial basis for $\Delta_{\boldsymbol{0}}$. Then the Vandermonde matrix
\[
M := \bigl( \delta_{\boldsymbol{0}} \circ P_i(D) \, \boldsymbol{x}^{\boldsymbol{\beta}_j} \bigr)_{1 \leq i,j \leq N}
\]
is nonsingular. Write each polynomial $P_i(\boldsymbol{x})$ ($1 \leq i \leq N$) as
\[
P_i(\boldsymbol{x}) = \sum_{\boldsymbol{x}^{\boldsymbol{\alpha}} \in \Lambda(P_i(\boldsymbol{x}))} e_{i\boldsymbol{\alpha}} \, \boldsymbol{x}^{\boldsymbol{\alpha}}, \qquad e_{i\boldsymbol{\alpha}} \in \mathbb{F},\ e_{i\boldsymbol{\alpha}} \neq 0.
\]
Then for any $\boldsymbol{x}^{\boldsymbol{\beta}_j} \in T$, we have
\[
\delta_{\boldsymbol{0}} \circ P_i(D) \, \boldsymbol{x}^{\boldsymbol{\beta}_j} = \sum_{\boldsymbol{x}^{\boldsymbol{\alpha}} \in \Lambda(P_i(\boldsymbol{x}))} e_{i\boldsymbol{\alpha}} \,\bigl( \delta_{\boldsymbol{0}} \circ D^{\boldsymbol{\alpha}} \, \boldsymbol{x}^{\boldsymbol{\beta}_j} \bigr).\]
Note that
\[
\delta_{\boldsymbol{0}} \circ D^{\boldsymbol{\alpha}} \, \boldsymbol{x}^{\boldsymbol{\beta}_j} =
\begin{cases}
\boldsymbol{\alpha}!, & \boldsymbol{\alpha} = \boldsymbol{\beta}_j, \\[4pt]
0, & \boldsymbol{\alpha} \neq \boldsymbol{\beta}_j.
\end{cases}
\]
%where $\boldsymbol{\alpha}! = \alpha_1! \, \alpha_2! \cdots \alpha_n!$. 
Consequently,
\[
\delta_{\boldsymbol{0}} \circ P_i(D) \, \boldsymbol{x}^{\boldsymbol{\beta}_j} = e_{i\boldsymbol{\beta}_j} \cdot \boldsymbol{\beta}_j!,
\]
with the convention that $e_{i\boldsymbol{\beta}_j} = 0$ if $\boldsymbol{\beta}_j \notin \Lambda\{P_i(\boldsymbol{x})\}$. Now suppose that there exists some $i_0$ such that
\[
\Lambda\{P_{i_0}(\boldsymbol{x})\} \cap T = \emptyset,
\]
i.e., $\boldsymbol{x}^{\boldsymbol{\beta}_j} \notin \Lambda\{P_{i_0}(\boldsymbol{x})\}$ for all $j = 1, 2, \dots, N$. Then $\delta_{\boldsymbol{0}} \circ P_{i_0}(D) \, \boldsymbol{x}^{\boldsymbol{\beta}_j} = 0$ for every $j$, which means that the $i_0$-th row of the matrix $M$ consists entirely of zeros. Hence $M$ is singular, contradicting the assumption that $T$ is an interpolation monomial basis for $\Delta_{\boldsymbol{0}}$.  
\end{proof}

The above lemma shows that the support of each $P_i(\boldsymbol{x})$ must have a nonempty intersection with the interpolation basis $T$.

For the single-node interpolation problem, suppose that the node is taken to be the origin and the corresponding interpolation basis has been determined. If the node is shifted to a nonzero point $\boldsymbol{z}$, the original interpolation basis may no longer constitute an interpolation basis for the new node. For example, consider the univariate interpolation problem with interpolation conditions
\[
\{\delta_z \circ (D_x - 1),\ \delta_z \circ D_x^2\},
\]
and take $T = \{x, x^2\}$.

When $z = 0$, it is easy to verify that the Vandermonde matrix is
\[
M = \begin{pmatrix}
1 & 0 \\
0 & 2
\end{pmatrix},
\]
which is clearly nonsingular. Hence $T$ is an interpolation monomial basis. However, when $z = 1$, the Vandermonde matrix becomes
\[
M = \begin{pmatrix}
0 & 1 \\
0 & 2
\end{pmatrix},
\]
which is singular. Consequently, $T$ no longer forms an interpolation monomial basis at this node.

If the interpolation node is not the origin, Lemma~\ref{lem:support_intersection} does not necessarily hold. For example, take the univariate node $z = 1$ with the interpolation conditions
\[
\delta_z \circ P_1(D),\quad \delta_z \circ P_2(D),
\]
where $P_1(x) = x$ and $P_2(x) = x^2$. Let $T = \{x^2, x^3\}$. Then the matrix
\[
M = \begin{pmatrix}
2 & 3 \\
2 & 6
\end{pmatrix}
\]
is nonsingular, so $T$ is an interpolation monomial basis. However, the support $\Lambda\{P_1(x)\}$ contains the monomial $x$, which does not belong to $T$.

Any given set of linearly independent polynomials $\{P_1(\boldsymbol{x}), P_2(\boldsymbol{x}), \dots, P_N(\boldsymbol{x})\}$ can be transformed by Gaussian elimination into an equivalent reverse reduced set $\{R_1(\boldsymbol{x}), R_2(\boldsymbol{x}), \dots, R_N(\boldsymbol{x})\}$. Note that each $R_i(\boldsymbol{x})$ is a linear combination of $P_1(\boldsymbol{x}), P_2(\boldsymbol{x}), \dots, P_N(\boldsymbol{x})$, hence
\[
\operatorname{span}\{P_1(\boldsymbol{x}), P_2(\boldsymbol{x}), \dots, P_N(\boldsymbol{x})\} = \operatorname{span}\{R_1(\boldsymbol{x}), R_2(\boldsymbol{x}), \dots, R_N(\boldsymbol{x})\}.
\]
Therefore, any given single-node Birkhoff interpolation problem (\ref{eq:2.3}) can be transformed into the equivalent form
\[
\Delta = \delta_{\boldsymbol{z}} \circ \operatorname{span}\{R_1(D), R_2(D), \dots, R_N(D)\},
\]
where $\{R_1(\boldsymbol{x}), R_2(\boldsymbol{x}), \dots, R_N(\boldsymbol{x})\}$ is a reverse reduced set.

For an arbitrary interpolation node, the analogue of Lemma~\ref{lem:support_intersection} is postponed to two corollaries below. We now establish the main result on the minimal monomial basis for the single-node Birkhoff interpolation problem.

\begin{theorem}
\label{thm:minimal_basis}%\ref{thm:minimal_basis}
For an arbitrary interpolation node $\boldsymbol{z} \in \mathbb{F}^n$, let a graded order $\prec$ and the Birkhoff interpolation conditions
\[
\Delta_{\boldsymbol{z}} = \delta_{\boldsymbol{z}} \circ \operatorname{span}\{P_1(D), P_2(D), \dots, P_N(D)\}
\]
be given. If $\{P_1(\boldsymbol{x}), P_2(\boldsymbol{x}), \dots, P_N(\boldsymbol{x})\}$ is a reverse reduced set w.r.t. $\prec$, then 
$\{ \operatorname{lm}(P_1(\boldsymbol{x})),\operatorname{lm}(P_2(\boldsymbol{x})),\dots, \operatorname{lm}(P_N(\boldsymbol{x}))\}$
%$\operatorname{lm}(P_j(\boldsymbol{x}))$ for $j = 1, 2, \dots, N$ 
forms a minimal monomial basis for $\Delta_{\boldsymbol{z}}$.
\end{theorem}

\begin{proof}
For the polynomials $P_j(\boldsymbol{x})$ ($1 \leq j \leq N$), without loss of generality we may assume that
\[
\operatorname{lm}(P_1(\boldsymbol{x})) \prec \operatorname{lm}(P_2(\boldsymbol{x})) \prec \cdots \prec \operatorname{lm}(P_N(\boldsymbol{x})).
\]
By formula (\ref{eq2}), we have $\delta_{\boldsymbol{z}} \circ P_j(D) = \delta_{\boldsymbol{0}} \circ e^{z D} P_j(D)$. Setting $\tilde{P}_j({\boldsymbol{x}}) = e^{\boldsymbol{x} \cdot\boldsymbol{z}} P_j(\boldsymbol{x})$, we obtain
\[
\delta_{\boldsymbol{z}} \circ P_j(D) = \delta_{\boldsymbol{0}} \circ \tilde{P}_j(D).
\]
Note that $\operatorname{lm}(\tilde{P}_j(\boldsymbol{x})) = \operatorname{lm}(P_j(\boldsymbol{x}))$, and
\[
\Lambda\{\tilde{P}_j(\boldsymbol{x})\} = \{\boldsymbol{x}^{\boldsymbol{\alpha} + \boldsymbol{\gamma}} \mid \boldsymbol{x}^{\boldsymbol{\alpha}} \in \Lambda\{P_j(\boldsymbol{x})\},\ \boldsymbol{\gamma} \in \mathbb{N}^n\}.
\]
Consider the matrix $M = (m_{ij})_{N \times N}$, where
\[
m_{ij} := \delta_{\boldsymbol{0}} \circ \tilde{P}_i(D) \bigl( \operatorname{lm}(P_j(\mathbf{\boldsymbol{x}})) \bigr), \quad 1 \leq i, j \leq N.
\]
\begin{enumerate}
    \item[(1)] For $i = j$, we have
    \[
    m_{ii} = \delta_{\boldsymbol{0}} \circ \tilde{P}_i(D) \bigl( \operatorname{lm}(P_i(\boldsymbol{x})) \bigr) \neq 0;
    \]
    \item[(2)] For $i > j$, we have $\operatorname{lm}(P_i(\boldsymbol{x})) \succ \operatorname{lm}(P_j(\boldsymbol{x}))$, and
    \[
    m_{ij} = \delta_{\boldsymbol{0}} \circ \tilde{P}_i(D) \bigl( \operatorname{lm}(P_j(\boldsymbol{x})) \bigr)
           = \delta_{\boldsymbol{z}} \circ P_i(D) \bigl( \operatorname{lm}(P_j(\boldsymbol{x})) \bigr)
           = 0.
    \]
\end{enumerate}

In summary, $M$ is an upper triangular matrix with nonzero diagonal entries, hence $M$ is nonsingular. This shows that
\[
T = \{\operatorname{lm}(P_1(\boldsymbol{x})), \operatorname{lm}(P_2(\boldsymbol{x})), \dots, \operatorname{lm}(P_N(\boldsymbol{x}))\}
\]
is an interpolation monomial basis for $\Delta_{\boldsymbol{z}}$. We now prove that $T$ is in fact a minimal monomial basis.

Suppose that we replace some $\operatorname{lm}(P_{j_0}(\boldsymbol{x}))$ in $T$, where $j_0 \in \{1, 2, \dots, N\}$, by a monomial $\boldsymbol{x}^{\boldsymbol{\beta}} \notin T$ with $\boldsymbol{x}^{\boldsymbol{\beta}} \prec \operatorname{lm}(P_{j_0}(\boldsymbol{x}))$. Let
\[
\tilde{T} = \{\boldsymbol{x}^{\boldsymbol{\beta}},\, \operatorname{lm}(P_1(\boldsymbol{x})), \operatorname{lm}(P_2(\boldsymbol{x})), \dots, \operatorname{lm}(P_{j_0-1}(\boldsymbol{x})), \operatorname{lm}(P_{j_0+1}(\boldsymbol{x})), \dots, \operatorname{lm}(P_N(\boldsymbol{x}))\}.
\]
We shall show that $\tilde{T}$ does not form a proper interpolation basis. Consider the Vandermonde matrix associated with $\tilde{T}$ and the interpolation conditions, namely,
\[
(\tilde{m}_{ij})_{N \times N} = 
\begin{cases} 
\delta_{\boldsymbol{z}} \circ P_i(D) \bigl( \operatorname{lm}(P_j(\boldsymbol{x})) \bigr), & 1 \leq i, j \leq N,\ j \neq j_0, \\[6pt]
\delta_{\boldsymbol{z}} \circ P_i(D) ( \boldsymbol{x}^{\boldsymbol{\beta}} ), & j = j_0.
\end{cases}
\]
Note that $\boldsymbol{x}^{\boldsymbol{\beta}} \prec \operatorname{lm}(P_{j_0}(\boldsymbol{x})) = \operatorname{lm}(\tilde{P}_{j_0}(\boldsymbol{x}))$. From
\[
\boldsymbol{x}^{\boldsymbol{\beta}} \prec \operatorname{lm}(P_{j_0+1}(\boldsymbol{x})) \prec \cdots \prec \operatorname{lm}(P_N(\boldsymbol{x})),
\]
it follows that
\[
\tilde{m}_{i,j_0} = \delta_{\boldsymbol{z}} \circ P_i(D)(\boldsymbol{x}^{\boldsymbol{\beta}}) = 0, \quad \forall i = j_0, j_0+1, \dots, N.
\]
The remaining columns of $\tilde{M}=(\tilde{m}_{ij})_{N \times N}$ are computed exactly as in the previous discussion. Hence $\tilde{M}$ is again an upper triangular matrix, but with
$\tilde{m}_{j_0,j_0} = 0,$
which implies that $\tilde{M}$ is singular. Therefore, $\tilde{T}$ does not form a proper monomial basis.
\end{proof}

According to Theorem~\ref{thm:minimal_basis}, for the Birkhoff interpolation problem with interpolation conditions given by \eqref{eq:2.3}, one only needs to compute the reverse reduced set of the polynomials appearing in the interpolation conditions. The least monomial of each polynomial in this reduced set then constitutes the minimal monomial basis w.r.t. $\prec$ for the interpolation problem. 
It is worth noting that the interpolation conditions of the single point Birkhoff interpolation problem are represented by an incidence matrix. The process of finding the reverse reduced set is equivalent to performing Gaussian elimination on the incidence matrix to obtain its reduced row echelon form. Hence we have the following theorem.

\begin{theorem}\label{pivot basis}
Let $\boldsymbol{z}$ be an interpolation node, and let $[\boldsymbol{x}^{\boldsymbol{\alpha}_1}, \boldsymbol{x}^{\boldsymbol{\alpha}_2}, \dots, \boldsymbol{x}^{\boldsymbol{\alpha}_l}]$ be a sequence of monomials arranged in a prescribed order $\prec$, with corresponding incidence matrix $E$. Performing Gaussian elimination on $E$ yields its reduced row echelon form. Then the monomials in the sequence corresponding to the pivot columns of this reduced form constitute the minimal monomial basis (with respect to the given order $\prec$) for the Birkhoff interpolation problem.
\end{theorem}

\begin{example}\label{ex2}
Given the interpolation node $\boldsymbol{z} = (1,1)$ and the sequence
\[
S = [y,\ y^2,\ xy,\ x^2,\ y^3],
\]
ordered according to the graded lexicographic order $y \prec_{\mathrm{grlex}} x$, the incidence matrix is
\[
E = 
\begin{pmatrix}
1 & 0 & 0 & 0 & 0 \\
1 & 1 & 0 & 1 & 0 \\
0 & 2 & 3 & 0 & 1
\end{pmatrix}.
\]
The Birkhoff interpolation conditions determined by $E$ can be written as
\[
\Delta = \delta_{(1,1)} \circ \operatorname{span}\{D_y,\ D_y^2 + D_x^2 + D_y,\ D_y^3 + 3D_xD_y + 2D_y^2\}.
\]
Reducing $E$ to its reduced row echelon form yields
\[
\begin{pmatrix}
1 & 0 & 0 & 0 & 0 \\
0 & 1 & 0 & 1 & 0 \\
0 & 0 & 1 & -\frac{2}{3} & \frac{1}{3}
\end{pmatrix}.
\]
The monomials corresponding to the pivot columns form the set $T = \{y,\ y^2,\ xy\}$. By Theorem~2, $\{y,\ y^2,\ xy\}$ is the minimal monomial basis for $\Delta$ with respect to the order $y \prec_{\mathrm{grlex}} x$.
\end{example}

\begin{remark}
For single-node Birkhoff interpolation, the traditional algorithm of Cui \cite{stable} computes the minimal interpolation basis by adding monomials one by one, constructing the Vandermonde matrix, and then performing row reduction. In contrast, the method proposed in this paper directly reduces the incidence matrix, which greatly reduces the computational cost.
\end{remark}

The following corollaries follow directly from Theorem~\ref{thm:minimal_basis}.
\begin{corollary}
\label{cor:independence_of_node}
The minimal monomial interpolation basis for the single point Birkhoff interpolation problem depends only on the differential conditions and is independent of the location of the node.
\end{corollary}

\begin{corollary}
\label{cor:support_intersection}
Let $\prec$ be a monomial order and $\boldsymbol{z} \in \mathbb{F}^n$ an interpolation node. Let
\begin{equation}\label{3.2}
 \Delta_{\boldsymbol{z}} = \delta_{\boldsymbol{z}} \circ \operatorname{span}\{P_1(D), P_2(D), \dots, P_N(D)\},   
\end{equation}
where the polynomials $P_i(\boldsymbol{x})$ are linearly independent. If $T$ is a $\prec$-minimal interpolation monomial basis for $\Delta_{\boldsymbol{z}}$, then for each
$P_i(\boldsymbol{x}), i = 1, 2, \dots, N,$
there exists a monomial $\boldsymbol{x}^{\boldsymbol{\alpha}_i} \in \Lambda\{P_i(\boldsymbol{x})\}$ such that $\boldsymbol{x}^{\boldsymbol{\alpha}_i} \in T$.
\end{corollary}

\begin{proof}
    The interpolation conditions \eqref{3.2} can be transformed into the equivalent form
\[
\Delta_{\boldsymbol{z}} = \delta_{\boldsymbol{z}} \circ \operatorname{span}\{R_1(D), R_2(D), \dots, R_N(D)\},
\]
where $\{R_1(\boldsymbol{x}), R_2(\boldsymbol{x}), \dots, R_N(\boldsymbol{x})\}$ is a reverse reduced set. Hence, by Theorem~1,
\[
T = \{\operatorname{lm}(R_1(\boldsymbol{x})), \operatorname{lm}(R_2(\boldsymbol{x})), \dots, \operatorname{lm}(R_N(\boldsymbol{x}))\}
\]
is the minimal monomial interpolation basis for $\Delta_{\boldsymbol{z}}$.
We now prove by contradiction that for every $P_i(\boldsymbol{x})$, $i = 1, 2, \dots, N$, there exists a monomial $\boldsymbol{x}^{\boldsymbol{\alpha}} \in \Lambda\{P_i(\boldsymbol{x})\} \cap T$. Suppose that there exists some $P_{t_0}(\boldsymbol{x})$ with $t_0 \in \{1, 2, \dots, N\}$ such that $\Lambda\{P_{t_0}(\boldsymbol{x})\} \cap T = \emptyset$. Since
\[
\operatorname{span}\{P_1(\boldsymbol{x}), P_2(\boldsymbol{x}), \dots, P_N(\boldsymbol{x})\} = \operatorname{span}\{R_1(\boldsymbol{x}), R_2(\boldsymbol{x}), \dots, R_N(\boldsymbol{x})\},
\]
the polynomial $P_{t_0}(\boldsymbol{x})$ can be expressed as a linear combination
\begin{equation}\label{eq3.5}
  P_{t_0}(\boldsymbol{x}) = c_1 R_1(\boldsymbol{x}) + c_2 R_2(\boldsymbol{x}) + \cdots + c_N R_N(\boldsymbol{x}),  
\end{equation}
where the coefficients $c_1, c_2, \dots, c_N$ are not all zero. Choose $k_0 \in \{1, 2, \dots, N\}$ such that $c_{k_0} \neq 0$. By the defining properties of a reverse reduced set, none of the polynomials
\[
R_1(\boldsymbol{x}), R_2(\boldsymbol{x}), \dots, R_{k_0-1}(\boldsymbol{x}), R_{k_0+1}(\boldsymbol{x}), \dots, R_N(\boldsymbol{x})
\]
contains the monomial $\operatorname{lm}(R_{k_0}(\boldsymbol{x}))$. Consequently, the term $c_{k_0} \operatorname{lm}(R_{k_0}(\boldsymbol{x}))$ appearing in $c_{k_0} R_{k_0}(\boldsymbol{x})$ cannot be canceled by any other summand on the right-hand side of \eqref{eq3.5}. Hence the right-hand side of \eqref{eq3.5} necessarily contains the monomial $\operatorname{lm}(R_{k_0}(\boldsymbol{x}))$. But $\operatorname{lm}(R_{k_0}(\boldsymbol{x})) \in T$, which contradicts the assumption that $\Lambda\{P_{t_0}(\boldsymbol{x})\} \cap T = \emptyset$.
\end{proof}

\section{An algorithm for computing the minimal monomial basis for the general Birkhoff interpolation problem}\label{sec:4}

For the general Birkhoff interpolation problem,  we use of the ring of formal power series to convert the interpolation conditions at multiple nonzero nodes into conditions at the origin. %The correspondence is summarized in Table~1, after which the subsequent computation proceeds.

The interpolation conditions of the multi-node Birkhoff interpolation problem can be written as
\begin{equation}
  \begin{aligned}\label{4.1}
&\delta_{\boldsymbol{z}_1} \circ \operatorname{span}\{P_{11}(D), P_{12}(D), \dots, P_{1s_1}(D)\}, \\
&\delta_{\boldsymbol{z}_2} \circ \operatorname{span}\{P_{21}(D), P_{22}(D), \dots, P_{2s_2}(D)\}, \\
&\qquad \vdots \\
&\delta_{\boldsymbol{z}_m} \circ \operatorname{span}\{P_{m1}(D), P_{m2}(D), \dots, P_{ms_m}(D)\}.
\end{aligned}  
\end{equation}
For $i = 1, 2, \dots, m$ and $j = 1, 2, \dots, s_i$, set
\begin{equation}\label{4.2}
    \tilde{P}_{ij}(\boldsymbol{x}) = e^{\boldsymbol{x} \cdot \boldsymbol{z}_i} P_{ij}(\boldsymbol{x}) = \sum_{k=0}^{\infty} \frac{(\boldsymbol{x} \cdot \boldsymbol{z}_i)^k}{k!} P_{ij}(\boldsymbol{x}) \triangleq \sum_{\boldsymbol{\alpha} \in \mathbb{N}^n} c_{i\boldsymbol{\alpha}} \, \boldsymbol{x}^{\boldsymbol{\alpha}}. 
\end{equation}

From the preceding discussion, the interpolation conditions (\ref{4.1}) are equivalent to the following conditions at the origin:
\[
\begin{aligned}
&\delta_{\boldsymbol{0}} \circ \operatorname{span}\{\tilde{P}_{11}(D), \tilde{P}_{12}(D), \dots, \tilde{P}_{1s_1}(D)\}, \\
&\delta_{\boldsymbol{0}} \circ \operatorname{span}\{\tilde{P}_{21}(D), \tilde{P}_{22}(D), \dots, \tilde{P}_{2s_2}(D)\}, \\
&\qquad \vdots \\
&\delta_{\boldsymbol{0}} \circ \operatorname{span}\{\tilde{P}_{m1}(D), \tilde{P}_{m2}(D), \dots, \tilde{P}_{ms_m}(D)\}.
\end{aligned}
\]
By Theorem~\ref{thm:minimal_basis}, we obtain the following conclusion for the multi-node Birkhoff interpolation problem.
\begin{theorem}
\label{thm:multivariate_minimal_basis}
Let
$\mathcal{R} = \{R_{11}(\boldsymbol{x}),R_{12}(\boldsymbol{x}), \dots, R_{1s_1}(\boldsymbol{x}), \dots, R_{m1}(\boldsymbol{x}), R_{m2}(\boldsymbol{x}), \dots, R_{ms_m}(\boldsymbol{x})\}$
be the reverse reduced set of the polynomials
\[
\{\tilde{P}_{11}(\boldsymbol{x}),\tilde{P}_{12}(\boldsymbol{x}), \dots, \tilde{P}_{1s_1}(\boldsymbol{x}), \dots, \tilde{P}_{m1}(\boldsymbol{x}),\tilde{P}_{m2}(\boldsymbol{x}), \dots, \tilde{P}_{ms_m}(\boldsymbol{x})\},
\]
where $\tilde{P}_{ij}(\boldsymbol{x}) = e^{\boldsymbol{x} \cdot \boldsymbol{z}_i} P_{ij}(\boldsymbol{x})$. Then the set of least monomials
\[
\{\operatorname{lm}(R_{11}(\boldsymbol{x})), \operatorname{lm}(R_{12}(\boldsymbol{x})), \dots, \operatorname{lm}(R_{1s_1}(\boldsymbol{x})), \dots, \operatorname{lm}(R_{m1}(\boldsymbol{x})), \operatorname{lm}(R_{m2}(\boldsymbol{x})),\dots, \operatorname{lm}(R_{ms_m}(\boldsymbol{x}))\}
\]
constitutes the minimal monomial interpolation basis w.r.t. $\prec$ for the multi-node Birkhoff interpolation problem \eqref{4.1}.
\end{theorem}

Theorem~\ref{thm:multivariate_minimal_basis} provides a concise theoretical framework for the general multivariate Birkhoff interpolation problem: one only needs to compute the reverse reduced set of the corresponding elements in the ring of formal power series and then take the least monomial of each element. This directly yields the desired minimal monomial basis. However, each $\tilde{P}_{ij}(\boldsymbol{x})$ defined by \eqref{4.2} is an infinite series, making it impossible to perform algebraic reduction directly within finitely many steps. Therefore, in actual computation, these series must first be suitably truncated. This converts the infinite-dimensional problem into a tractable one of row reduction on a finite matrix. In view of this, Algorithm~\ref{alg:truncated_set} below presents a truncation strategy that retains only finitely many candidate monomials. This lays the foundation for the subsequent reduction algorithm.

For the purpose of computational implementation, we adopt the following algorithm to retain only finitely many monomials.

\begin{algorithm}[h]
\caption{Computing the candidate monomial sequence $\tilde{S}$}
\label{alg:truncated_set}
\begin{algorithmic}
\State \parbox[t]{\linewidth}{\textbf{Input:} 
\begin{minipage}[t]{0.85\linewidth}
Interpolation polynomials
$P_{ij}(\boldsymbol{x}),  i = 1, 2, \dots, m; j = 1, 2, \dots, s_i;$\\ 
\hspace{1cm} \ graded order $\prec$.
\end{minipage}}
\State \textbf{Output:} Candidate monomial sequence $\tilde{S}$.
\State \textbf{Step 1: Relabel least monomials}
\State Arrange the least monomials $\operatorname{lm}(P_{ij}(\boldsymbol{x}))$ in increasing order with respect to $\prec$, and relabel them as
\State \hspace{\algorithmicindent} $\operatorname{lm}(P_1(\boldsymbol{x})) \prec \operatorname{lm}(P_2(\boldsymbol{x})) \prec \cdots \prec \operatorname{lm}(P_N(\boldsymbol{x})).$
\State \textbf{Step 2: Select candidate monomials}
\State From the support $\Lambda\{P_{ij}(\boldsymbol{x})\}$, select all monomials $\boldsymbol{x}^{\boldsymbol{\alpha}}$ satisfying
\State \hspace{\algorithmicindent} $\operatorname{lm}(P_1(\boldsymbol{x})) \preceq \boldsymbol{x}^{\boldsymbol{\alpha}} \preceq \operatorname{lm}(P_N(\boldsymbol{x})).$
\State $N^*:=$ the total number of such monomials.
\State Arrange these $N^*$ monomials in order $\prec$ and denote them as the monomial sequence $\tilde{S}:=[\boldsymbol{x}^{\boldsymbol{\alpha}_1}, \boldsymbol{x}^{\boldsymbol{\alpha}_2}, \dots, \boldsymbol{x}^{\boldsymbol{\alpha}_{N^*}}]$.
\If{$N^* < N$}
    %\State $\tilde{S} := [\boldsymbol{x}^{\boldsymbol{\alpha}_1}, \boldsymbol{x}^{\boldsymbol{\alpha}_2}, \dots, \boldsymbol{x}^{\boldsymbol{\alpha}_{N^*}}].$
%\Else
    \State Add to $\tilde{S}$ the $N-N^*$ monomials that following $\operatorname{lm}(P_N(\boldsymbol{x}))$ w.r.t. $\prec$.
\EndIf
\State \Return $\tilde{S}$
\end{algorithmic}
\end{algorithm}

\newpage
\begin{example}\label{ex3}
    Consider the interpolation nodes $\boldsymbol{z}_1 = (0, 0)$ and $\boldsymbol{z}_2 = (1, 1)$ from Example~\ref{ex1}. Given the sequence $S = [y, y^2, xy]$ ordered by the graded lexicographic order $y \prec_{\mathrm{grlex}} x$, the incidence matrix is $\boldsymbol{E} = \begin{pmatrix} E_1 \\ E_2 \end{pmatrix}$, where
\[
E_1 = \begin{pmatrix} 1 & 0 & 0 \\ 1 & 0 & 1 \end{pmatrix}, \quad
E_2 = \begin{pmatrix} 1 & 0 & 0 \\ 4 & 1 & 2 \end{pmatrix}.
\]

The corresponding sequence of differential operators is
\[
D = \left[ \frac{\partial}{\partial y},\ \frac{\partial^2}{\partial y^2},\ \frac{\partial^2}{\partial x \partial y} \right].
\]

Then the interpolation conditions $\Delta$ can be written as
\[
\Delta = 
\begin{cases} 
\delta_{(0,0)} \circ \operatorname{span}\{D_y,\ D_y + D_x D_y\}, \\[4pt]
\delta_{(1,1)} \circ \operatorname{span}\{D_y,\ 4D_y + D_y^2 + 2D_x D_y\}.
\end{cases}
\]

Set
\[
\tilde{P}_{11}(\boldsymbol{x}) = y,\  
\tilde{P}_{12}(\boldsymbol{x}) = y + xy,\  
\tilde{P}_{21}(\boldsymbol{x}) = y e^{x+y},\ 
\tilde{P}_{22}(\boldsymbol{x}) = (4y + y^2 + 2xy) e^{x+y},\]
%Expanding $e^{x+y}$ as a Taylor series gives\[e^{x+y} = 1 + (x+y) + \frac{(x+y)^2}{2!} +\frac{(x+y)^3}{3!} + \cdots.\]
where 
\[
\begin{aligned}
\tilde{P}_{21}(\boldsymbol{x}) &= y e^{x+y} = y + y^2 + xy + \frac{1}{2}y^3 + xy^2 + \frac{1}{2}x^2y + \cdots, \\[4pt]
\tilde{P}_{22}(\boldsymbol{x}) &= (4y + y^2 + 2xy) e^{x+y} = 4y + 5y^2 + 6xy + 3y^3 + 7xy^2 + 4x^2y + \frac{1}{2}y^4 + 2xy^3 + \frac{5}{2}x^2y^2 + x^3y + \cdots.
\end{aligned}
\]
Using Algorithm~\ref{alg:truncated_set}, we first obtain the set $\{y\}$ consisting of the least monomials of all $\tilde{P}_{ij}(\boldsymbol{x})$. Since this set contains only one element, we simply append the next three monomials following $y$ w.r.t. $y \prec_{\mathrm{grlex}} x$.
Finally, we form the sequence $\tilde{S} = [y, y^2, xy, y^3]$. Truncating each of
\[
\tilde{P}_{11}(\boldsymbol{x}),\ \tilde{P}_{12}(\boldsymbol{x}),\ \tilde{P}_{21}(\boldsymbol{x}),\ \tilde{P}_{22}(\boldsymbol{x})
\]
by retaining only the monomials that belong to $\tilde{S}$, we obtain the polynomials
\[
g_1(\boldsymbol{x}) := y,\ 
g_2(\boldsymbol{x}) := y + xy,\ 
g_3(\boldsymbol{x}) := y + y^2 + xy + \frac{1}{2}y^3,\ 
g_4(\boldsymbol{x}) := 4y + 5y^2 + 6xy + 3y^3.
\]

These can be written in matrix form as
\[
\begin{pmatrix}
g_1(\boldsymbol{x}) \\
g_2(\boldsymbol{x}) \\
g_3(\boldsymbol{x}) \\
g_4(\boldsymbol{x})
\end{pmatrix}
=
\begin{pmatrix}
1 & 0 & 0 & 0 \\
1 & 0 & 1 & 0 \\
1 & 1 & 1 & \frac{1}{2} \\
4 & 5 & 6 & 3
\end{pmatrix}
\begin{pmatrix}
y \\
y^2 \\
xy \\
y^3
\end{pmatrix}
\triangleq \tilde{E} \cdot \tilde{S}^{\mathrm{T}}.
\]
Performing Gaussian elimination on $\tilde{E}$ yields its reduced row echelon form
\[
\tilde{E}_{\text{rref}} =
\begin{pmatrix}
1 & 0 & 0 & 0 \\
0 & 1 & 0 & 0 \\
0 & 0 & 1 & 0 \\
0 & 0 & 0 & 1
\end{pmatrix}.
\]
By Theorem~\ref{pivot basis}, we obtain the minimal monomial interpolation basis w.r.t. $\prec$ for this example:
$\{y,\ y^2,\ xy,\ y^3\}$.
\end{example}
\begin{remark}
Example~\ref{ex3} shows that the main computational cost of our algorithm lies in the Gaussian elimination of the matrix $\tilde{E}$. Since $\tilde{E}$ can be easily obtained from the incidence matrix, the algorithm avoids the reliance on Vandermonde matrices required by traditional methods. Finally, we obtain the minimal monomial interpolation basis w.r.t. $\prec$ for the given interpolation conditions.
\end{remark}

The complete algorithm for reverse reduction and minimal interpolation basis computation is given as follows.
%\newpage
\begin{algorithm}[h]
\caption{Minimal interpolation basis for the general Birkhoff interpolation problem}
\label{alg:minimal_basis}
\begin{algorithmic}[1]
\State \parbox[t]{\linewidth}{\textbf{Input:} 
\begin{minipage}[t]{0.85\linewidth}
Interpolation nodes \( z_1, z_2, \dots, z_m \); \\
the associated polynomials \( P_{ij}(x), i = 1, 2, \dots, m; j = 1, 2, \dots, s_i \); \\
a graded order \( \prec \); \\
the candidate sequence \( \tilde{S} = [S_1, S_2, \dots, S_{N^*}] \) from Algorithm \ref{alg:truncated_set}.
\end{minipage}}

\State \textbf{Output:} The $\prec$-minimal monomial basis $T$.
\State $N := \sum_{i=1}^{m} s_i$;
\State  $\tilde{P}_{ij}(\boldsymbol{x}) := e^{\boldsymbol{x} \cdot \boldsymbol{z}_i} P_{ij}(\boldsymbol{x})$ by \eqref{4.2};
\State $d:=0$;
\State $\tilde{E}_{N \times N^*}=[\ ]$;
 \While {$d < N$} \do \\   
    \State $r := 1$;
    \For{$i = 1:m$}
        \For{$j = 1:s_i$}
            \State $g_r(\boldsymbol{x}) := \sum_{\boldsymbol{x}^{\boldsymbol{\alpha}} \in \tilde{S}} \operatorname{coeff}\bigl(\tilde{P}_{ij}(\boldsymbol{x}), \boldsymbol{x}^{\boldsymbol{\alpha}}\bigr) \cdot \boldsymbol{x}^{\boldsymbol{\alpha}}$;
            \State Fill row $r$ of $\tilde{E}$ with the coefficients of $g_r(\boldsymbol{x})$;
            \State $r := r + 1$;
        \EndFor
    \EndFor
    \State $\tilde{E}_{\text{rref}} := \text{rref}(\tilde{E})$(Compute the reduced row echelon form of $\tilde{E}$);
    \State $d := \operatorname{rank}(\tilde{E}_{\text{rref}})$;
    \If{$d<N$} 
        \State Add the $N - d$ monomials following $S_{N^*}$ w.r.t. $\prec$ into $\tilde{S}$;
        \State $N^*:=N^*+N-d;$
        \State $\tilde{E}:=[\tilde{E},\mathbf{0}_{N \times(N-d)}]$ (Extend matrix $\tilde{E}$ by appending  $N-d$ columns); 
    \EndIf
    \EndWhile
\State $T :=$ monomials corresponding to the pivot columns of $\tilde{E}_{\text{rref}}$;
\State \Return $T$.
\end{algorithmic}
\end{algorithm}

\begin{theorem}(Correctness of Algorithm~\ref{alg:minimal_basis})
For the Birkhoff interpolation problem with interpolation conditions \eqref{4.1}, Algorithm~\ref{alg:minimal_basis} terminates in finitely many steps, and the set $T$ of monomials it produces is precisely the minimal monomial basis.
\end{theorem}
\begin{proof}
The algorithm uses a graded order, which guarantees the finite termination of Algorithm~\ref{alg:minimal_basis}. When the algorithm terminates, the reduced row echelon form $\tilde{E}_{\mathrm{rref}}$ computed at line~17 has full rank $N$. At this point, each infinite series $\tilde{P}_{ij}$ has been truncated to the current candidate sequence $\tilde{S}$ at lines~11--12, and its coefficients have been inserted into the corresponding rows of the coefficient matrix $\tilde{E}$. After row reduction is performed on $\tilde{E}$ at line~16, the nonzero rows of its reduced row echelon form correspond to a reverse reduced set. The pivot columns precisely indicate the least monomials of the elements in this reverse reduced set. Accordingly, line~24 collects the monomials corresponding to the pivot columns to form the set $T$. By Theorem~\ref{thm:multivariate_minimal_basis}, the leading monomials of this reverse reduced set constitute the minimal monomial basis for the original multi-node Birkhoff interpolation problem. This establishes the correctness of Algorithm~\ref{alg:minimal_basis}.
\end{proof}

\begin{remark} 
In Example~\ref{ex3}, Algorithm~\ref{alg:minimal_basis} reaches $\operatorname{rank}(\tilde{E}_{\mathrm{rref}}) = N$ after the first execution of the loop (lines~7--17), so the minimal monomial interpolation basis $T$ is obtained immediately. If instead $d < N$ after line~17, one updates $\tilde{S}$, $N^*$, and augments the matrix as $\tilde{E} := [\tilde{E}, \boldsymbol{0}_{N \times (N-d)}]$. Remarkably, upon re-entering the loop at line~7, the previously computed row reduction of $\tilde{E}$ stays intact, and Gaussian elimination need only be applied to the newly appended columns.
\end{remark}
Finally, a numerical example is presented.

\begin{example}\label{ex4}
Consider the interpolation nodes $\boldsymbol{z}_1 = (0, 0)$ and $\boldsymbol{z}_2 = (1, 1)$. Given the monomial sequence
$S = [y,\ y^2,\ xy]$
ordered by the graded lexicographic order $y \prec_{\mathrm{grlex}} x$, the incidence matrix is $\boldsymbol{E}= \begin{pmatrix} E_1 \\ E_2 \end{pmatrix}$, where
\[
E_1 = \begin{pmatrix} 1 & 0 & 0 \\ 1 & 1 & 1 \end{pmatrix}, \qquad
E_2 = \begin{pmatrix} 1 & 0 & 0 \end{pmatrix}.
\]

The corresponding sequence of differential operators is
\[
D = \left[ \frac{\partial}{\partial y},\ \frac{\partial^2}{\partial y^2},\ \frac{\partial^2}{\partial x \partial y} \right].
\]

The interpolation conditions $\Delta$ can be written as
\[
\Delta =
\begin{cases}
\delta_{(0,0)} \circ \operatorname{span}\{ D_y,\ D_y + D_y^2 + D_x D_y \}, \\[6pt]
\delta_{(1,1)} \circ \operatorname{span}\{ D_y \}.
\end{cases}
\]
Algorithm~1 yields $\tilde{S} = [y, x, y^2]$. Algorithm~2 then gives
\[
\tilde{E} = 
\begin{pmatrix}
1 & 0 & 0 \\
1 & 0 & 1 \\
1 & 0 & 1
\end{pmatrix}.
\]
Row reduction shows that $\operatorname{rank}(\tilde{E}) = 2 < 3$. Hence the monomial $xy$ must be appended in the prescribed order, so that $\tilde{S} = [y, x, y^2, xy]$, and Algorithm~\ref{alg:minimal_basis} is repeated until $\operatorname{rank}(\tilde{E}) = 3$. Ultimately, Algorithm~\ref{alg:minimal_basis} produces
\[
\tilde{E} = 
\begin{pmatrix}
1 & 0 & 0 & 0 & 0 & 0 \\
1 & 0 & 1 & 1 & 0 & 0 \\
1 & 0 & 1 & 1 & 1 & 0 \\
1 & 0 & 1 & 1 & 0 & \frac{1}{2}
\end{pmatrix}, \quad
\tilde{E}_{\text{rref}} = 
\begin{pmatrix}
1 & 0 & 0 & 0 & 0 & 0 \\
0 & 0 & 1 & 1 & 0 & 0 \\
0 & 0 & 0 & 0 & 0 & 1
\end{pmatrix}.
\]
With respect to the order $y \prec_{\mathrm{grlex}} x$, the minimal monomial interpolation basis is $T = \{y, y^2, y^3\}$.

\end{example}

\begin{remark}
Example~\ref{ex4} shows that in some cases the set $\tilde{S}$ produced by Algorithm~\ref{alg:truncated_set} does not satisfy the proper condition. In such situations, additional monomials must be added before proceeding with the computation. That is, lines~19--21 of the algorithm are executed. 
\end{remark}

\section{Conclusion and future work}\label{sec:5}
Compared with ideal interpolation, the Birkhoff interpolation problem involves discontinuous derivative conditions at the nodes, which makes the computation of the corresponding monomial interpolation basis considerably more challenging. In this paper, we propose a new method for computing a proper monomial basis for Birkhoff interpolation, based on the properties of reverse reduced sets and the ring of formal power series. For the single-node case, the proper monomial basis is obtained directly by reducing the associated coefficient matrix to its reduced row echelon form. In the multi-node setting, we exploit the one-to-one correspondence between interpolation functionals and formal power series to convert the interpolation conditions at each nonzero node into conditions at the origin. The polynomials corresponding to the transformed conditions are then truncated to finitely many monomials, denoted by $g(\boldsymbol{x})$. Finally, Gaussian elimination is performed on the coefficient matrix of $g(\boldsymbol{x})$, yielding the minimal monomial interpolation basis with respect to the prescribed graded order.

This method does not require evaluating the interpolation functionals and avoids the use of Vandermonde matrices. Birkhoff interpolation suffers from poor numerical stability: a small perturbation of the nodes may change the minimal interpolation basis. In future work, we hope to exploit the theoretical framework developed in this paper to compute stable monomial bases for Birkhoff interpolation more conveniently.

\section*{Data availability}
No data was used for the research described in the article.

\section*{Acknowledgments}
This research was supported by the Department of Education of Jilin Province (JJKH20250468KJ), China.

\end{document}